\documentclass{amsart}
\usepackage{amichai}

\title{Small ideals in polynomial rings and applications}
\author{Amichai Lampert}
\thanks{The author is supported by the National Science Foundation under Grant No. DMS-1926686.}
\address{Institute for Advanced Study, Princeton, NJ, USA and University of Michigan, Ann Arbor, Mi, USA}
\email{amichai.lam@gmail.com}

\begin{document}

\maketitle

\begin{abstract}

Let $\k$ be a field which is either finite or algebraically closed and let $R = \k[x_1,\ldots,x_n].$ We prove that any $g_1,\ldots,g_s\in R$ homogeneous of positive degrees $\le d \le \ch(\k)$ are contained in an ideal generated by an $R_t$-sequence of $\le A(d)(s+t)^{B(d)}$ homogeneous polynomials of degree $\le d.$ When $\k$ is algebraically closed we can relax the degree restriction to $\ch(\k) =0 $ or $d\le\ch(\k)+1.$ This yields effective bounds for Ananyan and Hochster's theorem on strength and the codimension of the singular locus \cite[Theorem A]{AH} in new cases. It also implies effective bounds in the case $d=\ch(\k)$ for Tao and Ziegler's result \cite{inv-p} on rank of polynomials over finite fields and their $U^d$ Gowers norm.   
\end{abstract}

\section{Introduction}

In \cite{Sch}, Schmidt introduced the following notion of rank for polynomials:

\begin{definition}
    Let $\k$ be a field and $f\in\poly$ a homogeneous form. The \emph{Schmidt rank} of $f$ is
    \[
    \rk(f) = \min\left\{r: f = \sum_{i=1}^r g_ih_i\right\}
    \]
    where $g_i,h_i\in\poly$ are forms of lower degree. For a collection $f_1,\ldots,f_s\in\poly$ of forms of the same degree,
    \[
    \rk(f_1,\ldots,f_s) = \min \left\{\rk\left(\sum_{i=1}^s a_if_i\right): a_1,\ldots,a_s\in\k \textnormal{ not all zero}\right\}
    \]
\end{definition}

Schmidt obtained results on the asymptotic number of integer solutions to systems of rational polynomial equations, assuming those systems are of sufficiently large rank. These results were extended to several other settings in \cite{SK}, \cite{CM-primes}, \cite{Lee} and \cite{LZSchmidt}.

Independently of Schmidt's work, Green and Tao showed that high rank polynomials over finite fields are equidistributed as functions \cite{GT-bias}. This set off a flurry of research improving and extending their result in various ways, see \cite{BL}, \cite{KL08}, \cite{M-bias}, \cite{J-bias}, \cite{CM-bias} and \cite{MZ}. Recently, Schmidt rank\footnote{Ananyan and Hochster use the name \emph{strength}.} appeared as a crucial tool in Ananyan and Hochster's proof of Stillman's conjecture \cite{AH} and in a subsequent proof \cite{ESS}. 

In many of these results a regularization process is used in which an arbitrary collection of polynomials is replaced by a high rank collection which can then be analyzed. The number of high rank polynomials obtained through this process grows very quickly with the initial number of polynomials, and this is a significant barrier to obtaining reasonable bounds.

Our goal in this paper is to develop the notion of \emph{relative rank} for polynomials and to show that it retains many of the good properties of Schmidt rank while allowing for an efficient regularization process. Relative rank is a generalization of Schmidt rank and was introduced by the author and Ziegler in \cite{rel}, where collections of high relative rank over a finite field $\k$ were shown to be equidistributed as functions, as long as $\ch(\k)$ is larger than all the degrees involved. This has many consequences described in that paper, and can in fact be used to prove a similar result to theorem \ref{main} \cite[Theorem 2.6]{rel}. We will present a simpler, more direct proof, which also yields a result which is less restrictive on the degrees in positive characteristic. We now state our main technical result.

\begin{theorem}[Small ideals]\label{main}
    For every positive integer $d$ there exist constants $A(d),B(d)$ such that the following holds. Let $\k$ be a field and $g_1,\ldots,g_s\in\poly$ a collection of forms of positive degrees $ \le d.$ Suppose either of the following holds:
    \begin{itemize}
        \item $\k$ is algebraically closed of characteristic zero or $d\le \textnormal{char}(\k)+1,$ or   
        \item $\k$ is a finite field and $d\le \textnormal{char}(\k).$
    \end{itemize}

    Then there exist forms $G_1,\ldots,G_S\in \poly$ of degrees $\le d$ such that:
    \begin{enumerate}
        \item $g_1,\ldots,g_s \in I(G_1,\ldots,G_S),$
        \item  $S \le A(s+t)^B$ and
        \item $G_1,\ldots,G_S$ are an $R_t$-sequence.
    \end{enumerate}    
\end{theorem}

The hypotheses of the theorem naturally lead to two follow-up questions:
\begin{question}
    Is the restriction on the degrees in positive characteristic necessary?
\end{question}
\begin{question}
    Does the same result hold for any perfect field $\k?$
\end{question} 

Theorem \ref{main} has two immediate applications.

\subsection{Codimension of the singular locus}
Assume for now that $\k$ is algebraically closed. Given a homogeneous polynomial $f\in\poly,$ we set \linebreak
$S(f) = (x:\nabla f(x) = 0)$ . Similarly, for a collection of forms
\[
S(f_1,\ldots,f_s) = (x:\nabla f_1(x),\ldots,\nabla f_s(x) \text{ are linearly dependent}).
\]
Let $c(f_1,\ldots,f_s) := \codim_{\A^n} S(f_1,\ldots,f_s).$ A simple argument gives $\rk(f) \ge c(f)/2.$ Ananyan and Hochster proved \cite[Theorem A]{AH} that if $f$ has degree $d$ then \linebreak
$\rk(f)\le F(c(f),d)$ for some function $F$ and asked for reasonable bounds for $F.$  In \cite{KLP} the author, Kazhdan and Polishchuk proved the following effective version:

\begin{theorem}\label{cf-thm}
Let $\k$ be algebraically closed and assume that $\ch(\k)$ does not divide $d$. Then
\[
 \rk(f)\le (d-1)\cdot c(f).
\]
More generally, for a collection $f_1,\ldots,f_s$ of forms of degree $d$ we have
$$ \rk(f_1,\ldots,f_s)\le (d-1)\cdot (c(f_1,\ldots,f_s)+s-1).$$
\end{theorem}

Using theorem \ref{main} we obtain effective bounds in several more cases.  

\begin{theorem}[Rank and singularities]\label{rk-sing}
    Let $\k$ be an algebraically closed field of positive characteristic $p$ and $d\in\N$ such that either $d = p$ or $d =4$ and $p=2.$ Then there exist constants
    $A=A(d),B=B(d)$ such that for $f\in\poly$ homogeneous of degree $d$ we have  
    \[
    \rk(f) \le A(c(f))^B.
    \]
    More generally, for a collection of forms $f_1,\ldots,f_s\in\poly$ of common degree $d$ we have 
    \[
        \rk(f_1,\ldots,f_s) \le A(c(f_1,\ldots,f_s)+s-1)^B.
    \]
\end{theorem}

\subsection{Gowers norms}
The second application of theorem \ref{main} is to Gowers norms of polynomials over finite fields. Let $\k$ be a finite field and $f:\k^n\to\C$ some function. Given a vector $h\in\k^n,$ we define the (multiplicative) discrete derivative $\triangle_h f:\k^n\to\C$ by the formula $\triangle_h f(x) = f(x+h)\overline{f(x)}.$

\begin{definition}
    The \emph{Gowers uniformity norm} of $f$ is 
    \[
    \|f\|_{U^d} := |\E_{x,h_1,\ldots,h_d\in\k^n} \triangle_{h_1}\ldots\triangle_{h_d} f(x)|^{1/2^d},
    \]
    where $\E_{a\in A} = \frac{1}{|A|}\sum_{a\in A}$ for a finite set $A.$
\end{definition}

Gowers introduced these uniformity norms (for cyclic groups) in \cite{Gow} and used them in his analytic proof of Szemerédi's theorem on arithmetic progressions. They have since become a highly active subject of study with many applications, the most notable of which is Green and Tao's proof that the primes contain arbitrarily long arithmetic progressions \cite{GT-primes}. 

Let $\chi:\k\to\C$ be a non-trivial additive character. Given a polynomial \linebreak
$f\in\poly,$ we can compose with $\chi$ to obtain $\chi\circ f:\k^n\to\C.$ Beginning with the work of Green and Tao \cite{GT-bias}, and continuing through many improvements \cite{ShH}, \cite{BL}, \cite{M-bias}, \cite{J-bias}, \cite{CM-bias}, \cite{MZ}, it has been shown that if $f$ has degree $d<\ch(\k)$  and $\chi\circ f$ has large $U^d$ norm, then $f$ is low rank. In a breakthrough quantitative result \cite{M-bias}, Milićević proved:

\begin{theorem}
    Let $\k$ be a finite field, $\chi:\k\to\C$ a non-trivial additive character and $f\in\poly$ of degree $d<\ch(\k).$ There exist constants $A(d),B(d)$ such that if $\|\chi\circ f\|_{U^d} \ge |\k|^{-t}$ then 
    \[
    \rk(f) \le A(1+t)^B.
    \]
\end{theorem}

In \cite{inv-p}, Tao and Ziegler proved that an analogous inequality holds qualitatively when $d=\ch(\k),$ but so far no effective bounds have been obtained. By combining theorem \ref{main} with the results of \cite{rel}, we are able to obtain polynomial bounds in these cases.

\begin{theorem}[Large $U^p$ norm implies bounded rank]\label{bias-rk}
    Let $\k$ be a finite field of characteristic $p,$  $\chi:\k\to\C$ a non-trivial additive character and $f\in \poly$ a polynomial of degree $p.$ There exist constants $A=A(p),B=B(p)$ such that if $\|\chi\circ f\|_{U^p} \ge |\k|^{-t}$ then 
    \[
    \rk(f) \le A (1+t)^B.
    \]
\end{theorem}

The organization of the paper is as follows. In the next section we will use theorem \ref{main} to deduce theorems \ref{rk-sing} and \ref{bias-rk}. In section 3 we will define relative rank and state our main results about it for algebraically closed fields, and then in section 4 we will use them to prove theorem \ref{main} in the algebraically closed case. The remaining sections will be devoted to proving the results stated in section 3, before returning to finite fields in the final section.

\section{Deduction of theorems \ref{rk-sing} and \ref{bias-rk}} 

We start with a useful corollary of theorem \ref{main} which we will need for both proofs. This corollary follows by applying a clever argument from \cite[Proposition 2.6]{AH}, which we will use in several places throughout this paper.   

\begin{corollary}\label{Euler-p}
     Let $\k$ be a field of positive characteristic $p$ and $f\in\poly$ a form of degree $d.$ Suppose either 
     \begin{itemize}
         \item $p=2,d=4$ and $\k$ is algebraically closed or 
        \item $d=p$ and $\k$ is algebraically closed or finite.
     \end{itemize}
     There exist $C(d),D(d)$ such that if $g_1,\ldots,g_s$ are forms of positive degree $<d$ and $\frac{\partial f}{\partial x_i}\in I(g_1,\ldots,g_s)$ for all $i$ then $\rk(f)\le Cs^D.$ 
\end{corollary}

Note that when the characteristic doesn't divide the degree of $f$ this follows trivially from Euler's formula since $f = \frac{1}{d}\sum_{i=1}^n x_i \frac{\partial f}{\partial x_i} \in I(g_1,\ldots,g_s).$

\begin{proof}[Proof of corollary]
    Applying theorem \ref{main}, we obtain forms $G_1,\ldots,G_S$ which are an $R_3$-sequence such that $S\le A(s+3)^B$ and $\frac{\partial f}{\partial x_i}\in I(G_1,\ldots,G_S)$ for all $i.$ Set \linebreak
    $J:=I(G_1,\ldots,G_S).$ We now claim that the image of $f$ in $R:=\poly/J$ is reducible, which implies $\rk(f) \le A(s+3)^B+1.$ To get a contradiction, assume it's irreducible. $\poly/J$ is a unique factorization domain by a theorem of Grothendieck \cite[Corollaire XI 3.14]{SGA2} (see \cite{R3} for a simpler proof). Therefore, the image of $f$ in $R$ is prime, and in particular $R/(f)$ is reduced. Since $\k$ is perfect, we get that $\bar\k[x_1,\ldots,x_n]/I(G_1,\ldots,G_S,f)$ is reduced \cite[\href{https://stacks.math.columbia.edu/tag/030U}{Tag 030U}]{stacks}. Therefore $X = V(G_1,\ldots,G_S,f)$ is an algebraic set with $\codim_{\A^n} X =S+1$ and ideal $I(X) = I(G_1,\ldots,G_S,f).$ But the derivatives of $f$ vanish identically on $X$ so the tangent space at any given point is cut out by the derivatives of $G_i,$ and therefore has codimension $\le S.$ This contradicts smoothness of the variety $V(J+(f))$ in a nonempty open subset.   
\end{proof}

Let us now prove theorem \ref{rk-sing}. We begin by introducing some notation for the Taylor expansion around a point. If $f\in \poly$ and $x_0\in \A^n$ then we write 
\[
f(x+x_0) = \sum_{j\ge 0} f^j_{x_0}(x),
\]
where each $f^j_{x_0}(x)$ is homogeneous of degree $j.$ Note that this is really a finite sum. Here is a key lemma from \cite{KLP}, which we will also use later in this paper. 

\begin{lemma}\label{deriv-ideal}
Let $X\subset \A^n$ be an irreducible closed subvariety of codimension $c$ and $x_0\in X$ a smooth point. 
Let $g_1,\ldots,g_c$ be a set of elements in the ideal $I_X$ of $X$, with linearly independent differentials at $x_0$.
Then any homogeneous $f\in I_X$ of degree $d$  belongs to the ideal in $\poly$ generated by
$((g_i)^j_{x_0})_{i=1,\ldots,c;1\le j\le d}$. 
\end{lemma}

\begin{proof}[Proof of theorem \ref{rk-sing}]
    Write $c:= c(f).$ By the above lemma, we have $\frac{\partial f}{\partial x_i}\in I$ for all $i,$ where I is an ideal generated by $\le (d-1)c$ forms of positive degree $<d.$ By corollary \ref{Euler-p}, we get that $\rk(f)\le Ac^B.$  
\end{proof}

Now we move on to theorem \ref{bias-rk}. Let $V=\k^n$ and $F:V^p\to\k$ be the multilinear form (polarization) associated to $f,$ i.e. $F(h_1,\ldots,h_p) = \nabla_{h_1}\ldots\nabla_{h_p} f(0).$ For multilinear forms (tensors) there is a more nuanced version of rank.

\begin{definition}
    The \emph{partition rank} of $F,$ denoted $\prk(F),$ is the minimal $r$ with
    \[
    F = \sum_{i=1}^r G_i(x_{I_i}) H_i(x_{[p]\setminus I_i})
    \]
    where $\emptyset\neq I_i\subsetneq [p]$ and $G_i:V^{I_i}\to\k,H_i:V^{[p]\setminus I_i}\to\k$ are multilinear for all $i.$
\end{definition}

We use the following result of Milićević \cite{M-bias}.

\begin{theorem}
     Let $\k$ be a finite field, $\chi:\k\to\C$ a non-trivial additive character and $f\in\poly$ a homogeneous polynomial of degree $d.$ There exist constants $A(d),B(d)$ such that if $\|\chi\circ f\|_{U^d} \ge |\k|^{-t}$ then 
    \[
    \prk(F) \le A(1+t)^B.
    \]
\end{theorem}

\begin{proof}[Proof of theorem \ref{bias-rk}]
    By the above theorem, $\prk(F) \le A(1+t)^B.$ Let $F = \sum_{i=1}^r Q_i(x_{I_i}) R_i(x_{[p]\setminus I_i})$ with $p\not\in I_i$ for all $i$ and $r\le A(1+t)^B. $ Then for all $j$ we have $F(x_{[p-1]},e_j) \in I(Q_1,\ldots,Q_r).$ Plugging in the diagonal $x_{[p-1]} = (x,\ldots,x)$ we have the following for all $j$ 
    \[
    (p-1)! \partial_j f(x)=F(x,\ldots,x,e_j) \in I(Q_1(x,\ldots,x),\ldots,Q_r(x,\ldots,x)).
    \]
    By corollary \ref{Euler-p}, we get $\rk(f) \le At^B.$
\end{proof}

\section{Relative rank and singularities}

    We will define relative rank in the context of graded rings over a base field. The ones we will encounter are $\poly$ and its quotients by homogeneous ideals. Let $S = \bigoplus_{d\ge 0} S_d$ be a graded ring with $S_0=\k$ a field. For $f\in S$ homogeneous of positive degree, the definition of Schmidt rank naturally generalizes to
    \[
    \rk(f) = \inf\{r: f = \sum_{i=1}^r g_i\cdot h_i\}
    \]
    where $g_i,h_i$ are homogeneous of positive degree.


    \begin{definition}[Relative rank]
        Given a homogeneous ideal $I\subset S$ and a homogeneous element $f\in S$ of degree $d$, let $\bar f\in S/I$ be its image in the quotient ring. The \emph{relative rank} of $f$ on $I$ is 
        \[
        \rk_I(f) := \rk(\bar f) = \inf\{\rk(f+g): g\in I_d\}.
        \]
        For a collection of homogeneous elements $f_1,\ldots,f_s\in S_d,$ we define 
        \[
        \rk_I(f_1,\ldots,f_s) := \inf \{\rk_I(\sum_{i=1}^s a_if_i): a_1,\ldots,a_s\in \k \textnormal{ not all zero}\}.
        \]
    \end{definition}

    Note that we always have $\rk_I(f)\le \rk(f).$ If $I=(0)$ or is generated by homogeneous elements of degree $>d,$ then $\rk(f) = \rk_I(f).$
    
    We will be interested in collections of polynomials built up of layers with large relative rank. To describe this precisely, we need some definitions.

    \begin{definition}
		\begin{enumerate}
		  \item A \emph{tower} $\cF = (\cF_i)_{i\in [h]}$ of degree $\mathbf{d} = (d_1,\ldots,d_h)$ is composed of $h$ collections of homogeneous elements $\cF_i = (f_{i,j}) _{j\in[m_i]}$ which we call \emph{layers} such that the elements in each layer have common degree $d_i.$ The \emph{dimension} of the tower is $ |\cF| = m_1+\ldots+m_h. $ We denote the truncated tower $\cF_{<i} = (\cF_j)_{j\in [i-1]}.$ 
		\item The tower $ \cF $ is $(A,B,t)$-\emph{regular} if for every $i\in[h]$ we have  
			\[
			\rk_{I(\cF_{<i})} (\cF_i) > A(m_i+m_{i+1}+\ldots+m_h+t)^B. 
			\]
		\end{enumerate}
	\end{definition}

    Here is our main result connecting relative rank and the singular locus.
    \begin{theorem}[Relative rank and singularities]\label{rel-sing}
        Let $\k$ be an algebraically closed field and $\mathbf{d} = (d_1,\ldots,d_h),e$ a degree sequence such that either $\ch(\k) = 0$ or $\ch(\k) = p$ and $\mathbf{d}\le p, e \le p+2.$ There exist $A (\mathbf{d},e),B(\mathbf{d},e)$ such that whenever $\cF$ is an $(A,B,t)$-regular tower in $\poly$ of degree $\mathbf{d}$ and $f_1,\ldots,f_s$ are forms of degree $e$ with
        $\codim_{X}(S(f_1,\ldots,f_s,\cF)\cap X) \le t$ then we have 
        $$rk_{I(\cF)}(f_1,\ldots,f_s) \le A(t+s)^B.$$
    \end{theorem}
   
    In the next section we explain how theorem \ref{rel-sing} allows us to deduce theorem \ref{main}. Theorem \ref{rel-sing} implies that sufficiently regular towers satisfy Serre's condition $R_t.$ 
    
\begin{corollary}\label{reg-seq}
    Let $\cF$ be a tower in $\poly$ of degrees $\mathbf{d}=(d_1,\ldots,d_h),$ where either $\ch(\k)=0$ or $\ch(\k)=p$ and $d_i\le p, d_h\le p+2.$ Suppose $\cF$ is $(A,B,t)$-regular, where $A(d_1,\ldots,d_h),B(d_1,\ldots,d_h)$ are the constants from the above theorem. Then $\cF$ is an $R_t$-sequence.
\end{corollary}

\begin{proof}
    Since any subtower of $\cF$ is $(A,B,t)$-regular, it's enough that for all $i\in[h]$ we have
    \[
    \codim_{X(\cF_{<i})} S(\cF_{\le i})\cap X(\cF_{<i}) > t,
    \]
    and this is exactly the content of theorem \ref{rel-sing}.
\end{proof}

To prove theorem \ref{rel-sing} inductively we will need to work with a slightly more general version for forms which are bi-homogeneous. Given $f(x,y)$ bi-homogeneous we write $S_y(f) = ((x,y) : \frac{\partial f}{\partial y_i}(x,y) = 0\ \forall i).$ For a collection of bi-homogeneous functions $f_1,\ldots,f_s$ set 
\[
S_y(f) = ((x,y) : \rk\left(\frac{\partial f_i}{\partial y_j}\right)(x,y) < s).
\]

\begin{theorem}\label{rel-sing-y}
    Let $\k$ be an algebraically closed field. Let $\mathbf{d} = (d_1,\ldots,d_h),e$ such that either $\ch(\k) = 0$ or $\ch(\k) = p$ and $\mathbf{d},e\le p.$ There exist $A (\mathbf{d},e),B(\mathbf{d},e)$ such that the following holds. Let $\cF$ be a bi-homogeneous $(A,B,t)$-regular tower in $\k[x_1,\ldots,x_n,y_1,\ldots,y_n]$ of degree $\mathbf{d}$ with $\cF'\subset \cF$ the functions of positive degree in $y$ and let $f_1,\ldots,f_s$ be bi-homogeneous functions of degree $e$ and common bi-degree which is positive in $y.$ Suppose $ \codim_{X(\cF)} X(\cF)\cap S_y(f_1,\ldots,f_s,\cF') \le t.$ Then 
    $$\rk_{I(\cF)}(f_1,\ldots,f_s) \le A(t+s)^B.$$
\end{theorem}

 The reason we will need this generalization is for proposition \ref{Taylor-reg}. Note that for $e\le p,$ theorem \ref{rel-sing} is simply the case when there are no $x$ variables.

\section{Regularization and proof of theorem \ref{main}}

We now assume theorem \ref{rel-sing} holds and use it to deduce theorem $\ref{main}$ for algebraically closed fields. This will follow from an efficient relative regularization process, which is the content of the next lemma. This process will also play a key role later in our proof of theorem \ref{rel-sing}.

\begin{lemma}[Relative regularization]\label{regularize}
        Let $S = \bigoplus_{j\ge 0} S_j$ be a graded ring with $S_0 =\k$ a field. For any $A,B$ there exist constants $C(A,B,d),D(A,B,d)$ such that the following holds: If $(g_1,\ldots,g_s)$ is a collection of homogeneous elements of positive degree $\le d$  then there exists a tower $\mathcal{G}$ such that:
		\begin{enumerate}
            \item $I(g_1,\ldots,g_s)\subset I(\mathcal{G}).$
            \item $\mathcal{G}$ is of degree $(1,\ldots,d).$
		\item $\mathcal{G}$ is $(A,B,t)$-regular.
		\item $|\cG|\le C(s+t)^D.$
		\end{enumerate}
	\end{lemma}
\begin{proof}
    We begin by separating the $g_i$ into $d$ layers $\cG_i = (g_{i,j})_{j\in[m_i]}$ such that the $i$-th layer has degree $i.$ To each layer $i\in [d]$ we assign a number $n_i$ by a formula we'll soon describe. Now we regularize: If for each $i_0\in[d]$ we have $\rk_{I(\cG_{<i_0})} (\cG_{i_0}) > A(n_{i_0}+t)^B$ then we're done. Else, there exists some $i_0$ and $a_1,\ldots,a_{m_{i_0}}\in\k$ not all zero with 
    \[
    \sum_{j=1}^{m_{i_0}} a_j g_{i_0,j} = \sum_{k=1}^r p_kq_k \mod I(\mathcal{G}_{<i_0})
    \]
    with $r\le A(n_{i_0}+t)^B.$ Assume without loss of generality that $a_1\neq 0$ and replace $\mathcal{G}$ by the tower $\mathcal{G}'$ obtained by deleting $g_{i_0,1}$ and by adding each $p_k$ to the layer of degree $\deg(p_k)<i_0.$ Since $\cG_1$ is either linearly dependent or of infinite rank, this process must eventually terminate, yielding $\mathcal{G}$ satisfying requirements 1,2. Now we give the formula for $n_i$ in terms of $m_i,$ the original number of forms in each degree,
    \[
    n_d = m_d\ ,\ n_i = m_i+n_{i+1}A(n_{i+1}+t)^B.
    \]
    Note that $n_i$ bounds the total number of forms appearing in the layers $i,i+1,\ldots,d$ throughout the regularization process. This can be seen by downward induction, where the base case $i=d$ is clear. Assuming it holds for $i+1,$ we get that each of the $n_{i+1}$ forms appearing in degree $i+1$ and above can contribute at most $A(n_{i+1}+t)^B$ forms of degree $i,$ proving the claim for degree $i.$ This proves that requirement $3$ is satisfied. It's easy to see that all of the $n_i$ are polynomial in $s,t$ and so $\dim(\mathcal{G})\le n_1$ satisfies the fourth requirement.
\end{proof}

\begin{proof}[Proof of theorem \ref{main} for algebraically closed fields]
Let $A(1,\ldots,d),B(1,\ldots,d)$ be the constants of theorem \ref{rel-sing}. By the lemma, there exist $C(d),D(d)$ and $\cG$ of size $|\cG|\le C(s+t)^D$ which is $(A,B,t)$-regular with $g_1,\ldots,g_s\in I(\cG).$ By corollary \ref{reg-seq}, $\cG$ is an $R_t$-sequence as desired.
\end{proof}

In the next several sections, we will prove theorem \ref{rel-sing-y}.

\section{Low rank presentation on a Taylor tower}

From here until further notice assume $\k$ is algebraically closed. We now assume theorem \ref{rel-sing-y} (and therefore also theorem \ref{rel-sing}) holds for lower degrees and work towards proving it for the degree sequence $(d_1,\ldots,d_h,e).$ Since $$S_y(f_1,\ldots,f_s,\cF') = \bigcup_{a\in\P^s} S_y(\sum_i a_if_i,\cF')$$ it's enough to prove it for the case of a single polynomial $f.$ For a polynomial $g(x,y)$ we will use the notation $\partial_z g(x,y) = \nabla_{(0,z)} g(x,y)$ thinking of this as a polynomial in $(x,y,z).$ For a bi-homogeneous tower $\cG(x,y)$ we write $\partial_z \cG$ for a new tower in the variables $(x,y,z)$ with the same degrees in which $g\in \cG$ is replaced by $\partial_z g.$ Recall that given a bi-homogeneous tower $\cF(x,y),$ $\cF'\subset \cF$ is the subtower of functions of positive degree in $y.$ We write $T_z\cF$ for the tower $T_z\cF(x,y,z) = \cF\cup\partial_z \cF'.$ We begin with a useful relative analogue of lemma \ref{Euler-p}.

\begin{lemma}\label{Euler-rel}
    There exist constants $A,B$ depending on the degrees such that if $f,\cF$ are bi-homogeneous, $\cF$ is $(A,B,t)$-regular and  $\rk_{I(T_z\cF)} (\partial_z f) \le t,$ then we have $\rk_{I(\cF)} (f) \le At^B.$
\end{lemma}

\begin{proof}
    If $\ch(\k)$ doesn't divide the $y$-degree of $f$ then we can simply plug in $z=y$ and get $\rk_{I(\cF)} (f) \le t.$ The remaining case is when $\ch(\k)=p$ and $f=f(y)$ is of degree $p.$ By considering the double grading with respect to the variables $x,z,$ our assumption implies that $\partial_z f\in I(T_z\cF,g_1(y),\ldots,g_t(y)),$ and we may also assume without loss of generality that $\cF = \cF(y).$ Let $\bar g_1,\ldots,\bar g_t$ be the images of the $g_i$ in the ring $S = \k[y]/I(\cF).$ We apply lemma \ref{regularize} to $\bar g_i$ and obtain a tower $\bar \cG$ in $S$ which is $(A,B,3)$-regular. Let $\cG$ be any tower on $\k[y]$ of the same degrees as $\bar\cG$ with $\cG = \bar\cG \mod I(\cF).$ Note that if $\cF$ is $(A,B,3+|\cG|)$-regular, then by definition the tower given by stacking $\cG$ on top, $(\cF,\cG),$ is $(A,B,3)$-regular. This condition is satisfied so long as $\cF$ is $(C,D,t)$-regular for sufficiently large $C,D.$ The tower $(\cF,\cG)$ is an $R_3$-sequence by corollary \ref{reg-seq} and so $R = \k[y]/I(\cF,\cG)$ is a unique factorization domain by \cite{R3}. Suppose, to get a contradiction, that $f$ is irreducible in $R$ and hence a prime. But by construction $\nabla_z f(y)$ vanishes for any $(y,z)$ satisfying $y\in X(\cG), \nabla_z\cF(y) = 0$ but this contradicts the dimension of the tangent space to $X(f,\cF,\cG)$ being one smaller than that of $X(\cF,\cG).$
\end{proof}

Given a vector space $V$ over $\k,$ a point $v_0\in V$ and a polynomial $g\in \k[V]$ recall the Taylor expansion
\[
g(v_0+v) = \sum_{i\ge 0} g_{v_0}^i(v) 
\]
with $g_{v_0}^i(v)$ homogeneous of degree $i.$ For a tower $\cG$ we write $\cG_{v_0}^i$ for the collection of polynomials $g_{v_0}^i$ for $g\in\cG.$ We will later be more precise and organize the collection $(\cG_{v_0}^i)_{i\ge 0}$ into a tower. The following proposition is the first step towards the proof of theorem \ref{rel-sing-y} for the degree sequence $(d_1,\ldots,d_h,e).$ 

\begin{proposition}\label{low-rank-der}
    Suppose $\cF, f$ are as in theorem \ref{rel-sing-y} and set $r=(e-1)t.$ Then there exists a subvariety $Y\subset X(T_z\cF)$ of codimension $t$ such that for generic points $w_0\in Y$ we have homogeneous polynomials $h_1(x,y),\ldots,h_r(x,y)$ of positive degrees $<e$ with 
    \[
    \partial_z f\in I((T_z\cF)_{w_0}^j)_{j\ge 0}+I(h_1,\ldots,h_r).
    \]
\end{proposition}

\subsection*{Setup for proof of proposition \ref{low-rank-der}}
Let $\pi:X(T_z\cF)\to X(\cF)$ be the projection $\pi(x,y,z)=(x,y).$ Note that $\pi$ is surjective since $\pi(x,y,0)=(x,y).$ Let $Z\subset X(\cF)\cap S_y(f,\cF')$ be an irreducible component with $\codim_{X(\cF)} Z =t$ and choose an irreducible component $Y\subset \pi^{-1}(Z)$ with $\overline{\pi(Y)}=Z.$ We will show that the proposition holds for any such $Y.$ Before proving this, we need a couple of preliminary lemmas.

\begin{lemma}\label{tangent-rk}
    There exist constants $C,D$ such that if $\cF$ is $(C,D,t)$-regular then $T_z\cF$ is $(A,B,t)$-regular.
\end{lemma}

\begin{proof}
    Let $r_i = A(2m_i+\ldots+2m_h+t)^B$ be the desired relative rank of the $i$-th layer. Suppose that for some $i\in[h]$ we have a linear combination of elements of $(T_z\cF)_i$ which has relative rank $\le r_i$ on $(T_z\cF)_{<i},$
    \[
    \sum_j a_j f_{i,j}(x,y) + \sum_j b_j \partial_z f_{i,j}(x,y).
    \]
    By plugging in $z=0$ we get 
    \[
    \rk_{I(\cF_{<i})} \left(\sum_j a_j f_{i,j}(x,y)\right) \le r_i, 
    \]
    which implies $a_j = 0 $ for all $j.$ Now we're left with
    \[
    \rk_{I((T_z\cF)_{<i})} \left( \partial_z \left( \sum_j b_j f_{i,j}\right) \right) \le r_i,
    \]
    which by lemma \ref{Euler-rel} implies $\rk_{\cF_{<i}}(\sum_j b_j f_{i,j}) \le C(m_i+\ldots+m_h+t)^D, $ for sufficiently large $C,D.$ Therefore $b_j=0$ for all $j$ and we're done.
\end{proof}

\begin{lemma}\label{gen-dif}
    There exist $g_1(x,y),\ldots,g_t(x,y)\in I(Z)$ such that for generic points in $Y$ the differentials of $T_z\cF, g_1,\ldots,g_t$ are linearly independent. 
\end{lemma}

\begin{proof}
    By theorem \ref{rel-sing} the differentials of $\cF$ are linearly independent at generic points of $Z.$ If we choose some smooth point $q$ in this set, and choose $g_1,\ldots,g_t\in I(Z)$ such that the differentials of $\cF,g_1,\ldots,g_t$ at $q$ are linearly independent, then they are linearly independent for generic points in $Z.$ Thinking of these now as functions in $I(Y)\subset \k[x,y,z],$ all their partial derivatives with respect to the $z$ variables vanish. Now for $ f\in \cF',$ we have
    \[
    \nabla (\partial_z f (x,y)) = (*,\ldots,*,\frac{\partial f}{\partial y_1}(x,y),\ldots,\frac{\partial f}{\partial y_n}(x,y)).
    \]
    By theorem \ref{rel-sing-y}, these differentials are linearly independent in the $z$ coordinates for generic $(x,y)\in Z,$ so the entire collection $T_z\cF, g_1,\ldots,g_t$ has linearly independent differentials at points in $Y\cap \pi^{-1}(U)$ where $U\subset Z$ is a dense open subset. Since $\pi(Y)$ is dense in $Z,$ this is a dense open subset of $Y.$
\end{proof}

\begin{proof}[Proof of proposition \ref{low-rank-der}]
     We chose $Y$ to be an irreducible component of the intersection  $X(\partial_z\cF')\cap \{(x,y,z):(x,y)\in Z\}.$ By the principal ideal theorem these varieties have codimension $\le |\partial_z\cF'|$ and $\le |\cF|+t,$ so the codimension of $Y$ in affine space is at most $|T_z\cF|+t.$ Lemma \ref{gen-dif} shows that the opposite inequality holds as well so the codimension is exactly $|T_z\cF|+t.$ By lemma \ref{tangent-rk} and corollary \ref{reg-seq}, $T_z\cF$ is a prime sequence so $\codim_{X(T_z\cF)} Y = t.$ Now note that $\partial_z f\in I(Y).$ This is because theorem \ref{rel-sing-y} implies that $U := Z\setminus S_y(\cF')$ is a nonempty open subset and so $Y\cap \pi^{-1}(U)$ is a dense open subset of $Y.$ The function $\partial_z f$ vanishes on this dense open subset, hence it vanishes identically on $Y.$ Combining lemma \ref{gen-dif} with lemma \ref{deriv-ideal} we get that for generic points $w_0\in Y$ we have 
     \[
        \partial_z f(x,y)\in I((T_z\cF)_{w_0}^j)_{j\in[e]}+I((g_i)^j_{w_0})_{i\in[t],j\in [e-1]}.
     \]
     Why does $j$ run up to $e-1$ and not $e$? This is because the $g_i$ are functions of $(x,y)$ only and $\partial_z f(x,y)$ has degree $e-1$ in $(x,y).$ Therefore any contribution from $(g_i)^e_{w_0}$ must cancel out. This completes the proof of the proposition.
\end{proof}

The remaining obstacle is to remove the elements of the Taylor expansion of $T_z\cF.$ To do this, we will use the fact that the above holds for \emph{generic} points $w_0\in Y.$ It turns out that by combining the low rank representations obtained for different basepoints we can get rid of the unwanted terms.

\section{The Taylor tower is generically regular}

The goal of this (somewhat technical) section is to prove that when we start with a sufficiently regular tower $\cG\subset \k[V]$ (in our case $T_z\cF$) and choose a generic tuple of points in $X(\cG),$ then the tower $\cG$ along with the elements of its Taylor expansion at these points is regular. Before we state the proposition making this precise, we will need some notation and a convention for how to organize the tower of polynomials in the Taylor expansion. Let $g\in\k[V]$ be homogeneous of degree $d.$ The polynomial $g(w+v)\in\k[V\oplus V]$ has Taylor expansion
\[
g(w+v) = g^0(w,v)+g^1(w,v)+\ldots+g^d(w,v)
\]
where $g^i$ each has total degree $d$ and is bi-homogeneous of degree $(d-i,i).$ Note that $g^d(w,v) = g(v).$ Given a tower $\cG = (g_{i,j})_{i\in[h],j\in[m_i]}\subset\k[V]$ of homogeneous forms we define $D\cG\subset \k[V\oplus V]$ to be the bi-homogeneous tower given by \linebreak
$(D\cG)_i = (g^k_{i,j}(w,v))_{j\in[m_i],0\le k\le d_i},$ where $d_i = \deg(\cG_i).$ For a positive integer $l,$ we set $D_l\cG \subset\k[V^l\oplus V]$ to be the bi-homogeneous tower with layers
$$(D_l\cG)_i = (g_{i,j}(v))_{j\in[m_i]} \cup (g^k_{i,j}(w_\lambda,v))_{j\in[m_i],0\le k\le d_i-1,\lambda\in[l]}.$$

For forms $\cG = (g_1,\ldots,g_s)$ in $\k[V]$ of degree $d$ and fixed $w \in V^l$ we obtain a tower $\cG(w)$ on $\k[V]$ of degree $(1,2,\ldots,d)$ with top layer $(g_1(v),\ldots,g_s(v))$ and $i$-th layer  $(g_j^i(w_\lambda,v))_{j\in[s],\lambda\in[l]}.$ Doing this for each of the layers of $\cG$ we obtain a tower $D_l\cG(w)$ on $\k[V]$ of degree $(1,2,\ldots,d_1,1,2,\ldots,d_2,\ldots,1,\ldots,d_h)$ with $d_1+\ldots+d_r$-th layer $\cG_r$ and $d_1+\ldots+d_{r-1}+i$-th layer $(\cG_r(w))_i.$ 
Note that the polynomials in this tower are the collection $(\cG^k_{w_\lambda})_{k\ge 0,\lambda\in[l]}.$
We now have sufficient notation to state the central proposition of this section, which is a version for relative rank of \cite[Theorem 1.6]{KLP}.

\begin{proposition}\label{Taylor-reg}
     Let $\textbf{d}=(d_1,\ldots,d_h)$ be a sequence of degrees such that either $\k$ has characteristic $0$ or $d_i\le \ch(\k)$ for all $i.$  Given $A,B$ there exist constants $C(\textbf{d},A,B,l),D(\textbf{d},A,B,l)$ such that if $\cG$ is a $(C,D,t)$-regular tower of degree $\textbf{d}$ then for generic $w\in V(\cG)^l$ outside an algebraic set of codimension $t,$ the tower $D_l\cG(w)$ is $(A,B,t)$-regular.
\end{proposition}

\textbf{Note:} The restriction on the degree in positive characteristic cannot be relaxed and this is the main obstacle to obtaining Theorem \ref{main} for higher degrees. We include a brief discussion at the end of this section.

To prepare for the proof, we start with a few lemmas. If $g\in\k[V]$ is homogeneous of degree $d,$ then associated to it is a multi-linear form $\tilde g:V^d\to\k$ given by $\tilde g(v_1,\ldots,v_d) = \nabla_{v_1}\ldots\nabla_{v_d} g(0).$ Note that $\tilde g(x,\ldots,x) = d!g(x).$ If $\cG$ is a tower with highest degree $d,$ we get a tower of forms $\tilde \cG$ defined on $V^d.$ For each function $g$ in $\cG_i$ of degree $d_i$ and subset $E\in\binom{[d]}{d_i},$ $\tilde \cG_i$ contains a function $ g_E(v_1,\ldots,v_d) = \tilde g((v_i)_{i\in E}). $

\begin{lemma}\label{ml-reg}
    If the degree of $\cG$ and characteristic of $\k$ are as in proposition \ref{Taylor-reg} and $A,B$ are constants, then there exist constants $C(\textbf{d},A,B),D(\textbf{d},A,B)$ such that if $\cG$ is $(C,D,t)$-regular, then $\tilde\cG$ is $(A,B,t)$-regular.
\end{lemma}

\begin{proof}
    Let $d$ be the maximal degree in $\cG.$ By induction on the degree sequence, we may assume that $\tilde \cG_{<h}$ is $(A,B,t+2^dm_h)$-regular. Write $r =A(t+2^dm_h)^B$ and assume, to get a contradiction, that $\rk_{I(\tilde \cG_{<h})} (\tilde\cG_h) \le r. $ Then there exists a non-trivial linear combination 
    \begin{equation}\label{low-rk-comb}
        g = \sum_{E\in\binom{[d]}{d_h}}\sum_{j\in[m_h]} a_{j,E}(g_{h,j})_E
    \end{equation}
    with $\rk_{I(\tilde \cG_{<h})} (g) \le r,$ where $d$ is the maximal degree in $\cG.$ First we handle the case where $\ch(\k) = 0$ or $d_h<\ch(\k).$ Choose $E_0\in \binom{[d]}{d_h}$ such that $a_{j,E_0}$ are not all zero. Plugging into equation \eqref{low-rk-comb} $v_k = x$ for $k\in E_0$ and $v_k = 0$ for $k\not\in E_0,$ we obtain $\rk_{I(\cG_{<h})} (\cG_h) \le r.$ This is impossible when $C,D$ are sufficiently large. The remaining case is when $\ch(\k)=d_h=p.$ In this case we plug $v_1=\ldots=v_{p-1}=x$ and $v_p =z$ into equation \eqref{low-rk-comb}. This yields $\rk_{I(T_z\cG)} (\nabla_z \cG_h) \le r.$ Assuming without loss of generality that $A,B$ are at least as large as the constants in lemma \ref{Euler-rel}, we apply it to get $\rk_{I( \cG_{<h})} (\cG_h) \le Ar^B,$ which is a contradiction when $C,D$ are sufficiently large.
\end{proof}

\begin{lemma}\label{Dl-reg}
    If the degree of $\cG$ and characteristic of $\k$ are as in proposition \ref{Taylor-reg}, and $A,B,l\in\N,$ there exist constants $C(\textbf{d},A,B,l),D(\textbf{d},A,B,l)$ such that if $\cG$ is $(C,D,t)$-regular, then $D_l\cG$ is $(A,B,t)$-regular.
\end{lemma}

\begin{proof}
    Let $d$ be the maximal degree in $\cG$ and assume by induction that $D_l\cG_{<h}$ is $(A,B,t+(1+dl)m_h)$-regular. Set $r = A(t+(1+dl)m_h)^B$ and suppose, to get a contradiction, that $\rk_{I(D_l\cG_{<h})} (D_l\cG_h) \le r.$ Then there is a non-trivial linear combination
     \begin{equation}\label{Dl-low-rk}
         g = \sum_{j\in[m_h]} a_j g_{h,j}(v)+\sum_{j\in[m_h],0\le k\le d_h-1,\lambda\in[l]} b^k_{j,\lambda} g_{h,j}^k(w_\lambda,v)
     \end{equation}
     with $\rk_{I(D_l\cG_{<h})} (g) \le r.$ If the $a_j$ are not all zero, then we can plug in $w_\lambda =0$ for all $\lambda$ and then equation \eqref{Dl-low-rk} yields $\rk_{I(\cG_{<h})} (\cG_h) \le r,$ a contradiction. If $a_j =0$ for all $j$, there is some $\lambda_0\in[l]$ for which $b_{j,\lambda_0}^k$ are not all zero. We plug in $w_{\lambda_0} = w$ and $w_\lambda = 0$ for all $\lambda\neq \lambda_0.$ Thus $g$ specializes to 
     \[
     g'(w,v) = \sum_{j\in[m_h],0\le k\le d_h-1} b^k_{j,\lambda_0} g_{h,j}^k(w,v),
     \]
     and $I(D_l\cG_{<h})$ specializes to $I(D\cG_{<h}).$
     Now choose $0\le k_0\le d_h-1$ such that $b^{k_0}_{j,\lambda_0}$ are not all zero and apply $\nabla_{(0,v_1)}\ldots\nabla_{(0,v_{k_0})}\nabla_{(w_1,0)}\ldots\nabla_{(w_{d_h-k_0},0)}.$ This yields $\rk_{I(\tilde\cG_{<h})} (\sum_{j\in[m_h]} b^{k_0}_{j,\lambda_0} \tilde g_j)\le 2^{d_h}r. $ By lemma \ref{ml-reg} this gives a contradiction when $C,D$ are sufficiently large.
\end{proof}

\begin{lemma}\label{low-rk-sing}
    Suppose $\cG = (g_1,\ldots,g_s)\in\k[V]$ are homogeneous forms and \linebreak
    $\rk_{I(\cG)}(f) \le r. $ Then 
    \[
    \dim X(\cG)\cap S(f,\cG) \ge \dim X(\cG)-2r.
    \]
\end{lemma}

\begin{proof}
    Write $f(x) = \sum_{j=1}^r p_j(x) q_j(x) \mod (I(\cG)).$ The directional derivative satisfies
    \[
    \nabla_v f(x) = \sum_{j=1}^r p_j(x) \nabla_v q_j(x)+ \nabla_v p_j(x) q_j(x) \mod (g_i(x), \nabla_v g_i(x)),
    \]
    so $X(\cG)\cap (p_j = q_j =0) \subset X(\cG)\cap S(f,\cG).$
    Let $X_0\subset X(\cG)$ be an irreducible component of dimension $\dim X(\cG).$
    Then every irreducible component of  the nonempty algebraic set $X_0 \cap (p_j = q_j = 0)$ has dimension $\ge \dim(X_0) -2r.$
\end{proof}

We need one more well-known lemma, which we prove for completeness.

\begin{lemma}\label{Markov}
    Let $\varphi:X\to W$ be a morphism of affine algebraic sets (not necessarily irreducible). Let $W_\eta$ be the Zariski closure of
    $$\{w\in W: \dim \varphi^{-1}(w)\ge \dim X-\dim W+\eta\}.$$ 
    Then $\dim W_\eta \le \dim W - \eta.$
\end{lemma}

\begin{proof}
    Without loss of generality we may assume that $X$ is irreducible and $\varphi$ is dominant. Let $X_\eta = \varphi^{-1}(W_\eta).$ The map
    $\psi = \varphi\restriction_{X_\eta}:X_\eta\to W_\eta$ is dominant so for generic points $w\in W_\eta$ we have $\dim \varphi^{-1}(w) = \dim \psi^{-1}(w) = \dim X_\eta - \dim W_\eta$ (this is \cite[Theorem 9.9]{Milne}, for example). Therefore,
    \[
    \dim X-\dim W+\eta \le \dim \varphi^{-1} (w) = \dim X_\eta -\dim W_\eta \le \dim X -\dim W_\eta.
    \]
    Rearranging, we get $\dim W-\dim W_\eta \ge \eta.$
\end{proof}

\begin{proof}[Proof of proposition \ref{Taylor-reg}]
     By induction on the degree sequence, for generic \linebreak $w\in X(\cG_{<h})^l$ outside an algebraic set of codimension $t+lm_h$ we have that $D_l\cG_{<h}(w)$ is $(A,B,t+lm_hd_h))$-regular. By interesecting, this holds for $w\in X(\cG)^l$ outside an algebraic set of codimension $t.$ We can finish by showing that for $w\in X(\cG_{<h})^l$ outside an algebraic set of codimension $t+lm_h$ we have
     \[
     \rk_{I(D_l\cG_{<h}(w))} D_l\cG_h(w) > A(t+lm_hd_h)^B.
     \]
     Set $r = A(t+lm_hd_h)^B$ and let
    \[
    X \vcentcolon = X(D_l\cG_{<h}),\ Z\vcentcolon = X\cap  S_v(D_l\cG),\ 
    W \vcentcolon = X(\cG_{<h})^l.
    \]
    If $C,D$ are sufficiently large then we can apply theorem \ref{rel-sing-y} to $D_l\cG_h$ on the tower $D_l\cG_{<h}$ to get $\codim_X Z  > 3r.$ Let  $\varphi:X\to W$ be the projection on the $w$ coordinates and write $X(w),Z(w)$ for $\varphi^{-1}(w), \varphi^{-1}(w)\cap Z$ respectively. By corollary \ref{reg-seq}, $X,W$ are irreducible complete intersections. For any $w\in W,$ the principal ideal theorem yields
    \begin{align*}
        \dim X(w) &\ge \dim V -|D_l\cG_{<h}(w)| = \dim V - (|D_l\cG_{<h}| - |\cG_{<h}|^l) \\
        &= \dim (V^l\oplus V) - |D_l\cG_{<h}| - (\dim V^l - |\cG_{<h}|^l) \\
        & = \dim X - \dim W.
    \end{align*}
     By lemma \ref{Markov} the set of $w\in W$ with
    $\dim(Z(w)) \ge \dim Z - \dim W + r$ is contained in an algebraic set $W_r$ with $\codim_W (W_r) \ge r.$ For any $w\in W\setminus W_r$ we have
    \[
    \dim Z(w) < \dim Z-\dim W+r = \dim X-\dim W-2r \le \dim X(w)-2r.
    \]
    Note that $Z(w) = S(D_l\cG(w))\cap X(D_l\cG_{<h}(w))$ so by lemma \ref{low-rk-sing} we have \linebreak
    $\rk_{I(D_l\cG_{<h}(w))} D_l\cG_h(w) > r$ for these $w,$ completing the proof.   
\end{proof}

\subsection{Necessity of the degree bound in Proposition \ref{Taylor-reg}} Proposition \ref{Taylor-reg} doesn't hold when $\k$ has characteristic $p\ge 3$ and the tower $\cF$ contains forms of degree $p+1.$ For example, the form $f(v) = \sum_{i=1}^n v_i^{p+1}$ has $\rk(f) \ge n/2$ and $f^2(w,v) = 0.$ This limits our proof method and prevents us from obtaining theorem \ref{main} for higher degrees.

Now we turn to characteristic $2.$ In degree $3,$ the results of \cite{KLP} imply \linebreak
$\rk(f^i) \ge \rk(f)/4$ for any $0\le i\le 3.$ In degree $5,$  Proposition \ref{Taylor-reg} fails for the form \linebreak 
$g(v) = \sum_{i=1}^n v_i^5$ which has $\rk(g) \ge n/2$ and  $g^2(w,v) = 0.$ What happens in degree $4$ appears to be an interesting question.
\begin{question}
     Does there exist a function $\Gamma$ such that whenever $\k$ is an algebraically closed field of characteristic $2$ and
    $f\in\k[V]$ is a homogeneous polynomial of degree $4,$ we have
    \[
    \rk(f) \le \Gamma (\rk(f^2(w,v))) ?
    \]
\end{question}

\section{Completing the proof of theorems \ref{rel-sing-y} and theorem \ref{rel-sing}}

    We combine the results of the previous three sections to obtain:

    \begin{proposition}\label{gluing}
        Under the assumptions of theorem \ref{rel-sing-y}
        there exist constants $F(\textbf{d},e),G(\textbf{d},e)$ and collections of forms of positive degrees   $\cG_1,\ldots,\cG_{e+1}$ and $\cG = (g_1,\ldots,g_\tau)$ such that 
         \begin{enumerate}
            \item $\cG$ depends only on $(x,y)$ and is of degree $<e,$
            \item $\partial_z f\in I(\cG_i\cup T_z\cF\cup\cG)$ for all $i\in[e+1],$
             \item $\cup_i\cG_i\cup T_z\cF\cup\cG$ is a regular sequence and
             \item $\tau \le Ft^G.$
         \end{enumerate}
    \end{proposition}

\begin{proof}
    By proposition \ref{low-rank-der} we have a subvariety $Y\subset X(T_z\cF)$ of codimension $t$ such that for generic points $w_0\in Y$ there are homogeneous polynomials $h^{w_0}_1(x,y),\ldots,h^{w_0}_r(x,y)$ of positive degrees $<e$ with 
    \begin{equation}\label{rk-cond}
         \partial_z f\in I((DT_z\cF)(w_0))+I(h^{w_0}_1,\ldots,h^{w_0}_r),
    \end{equation}
    where $r=(e-1)t.$
    The subvariety $Y^{e+1}\subset X(T_z\cF)^{e+1}$ is of codimension $t(e+1).$ Let $A'(\textbf{d},e),B'(\textbf{d},e)$ be constants to be chosen later. By combining lemma \ref{tangent-rk} and proposition \ref{Taylor-reg}, there exist constants $C(A',B',\textbf{d},e),D(A',B',\textbf{d},e)$ such that if $\cF$ is $(C,D,t)$-regular then for a generic point $w\in X(T_z\cF)^{e+1}$ outside a subvariety of codimension $t(e+1)+1,$ the tower $(D_{e+1}T_z\cF)(w)$ is $(A',B',t)$-regular. Therefore we can choose $w\in Y^{e+1}$ such that both \eqref{rk-cond} and our condition on the regularity of $(D_{e+1}T_z\cF)(w)$ hold. Set $\cG_i = (DT_z\cF)(w_i) \setminus T_z\cF $ and let $\cG' = (h^{w_\alpha}_\beta)_{\alpha\in[e+1],\beta\in[r]}.$ Note that the first two conditions of the proposition hold with $\cG'$ in the role of $\cG,$ but not necessarily the third. Let $S = \k[\bar x,\bar y,\bar z]/I((D_{e+1}T_z\cF)(w))$ and let $A,B$ be the constants of theorem \ref{rel-sing} for the degree sequence of $(D_{e+1}T_z\cF)(w)$ followed by $(1,\ldots,e-1).$ Applying lemma \ref{regularize} to the image of $\cG'$ in the graded ring $S$ we obtain a tower $\cG$ of size $\le Ft^G$ which is $(A,B,1)$-regular, where $F(A,B,e),G(A,B,e)$ are constants. Note that $\cG$ is still composed of functions of $(x,y)$ since functions of positive degree in $z$ don't contribute to the low rank representations appearing in the process described by lemma \ref{regularize}. Now if $A'(\textbf{d},e),B'(\textbf{d},e)$ are chosen sufficiently large then $(D_{e+1}T_z\cF)(w)$ is $(A,B,1+Ft^G)$-regular, so the tower obtained by adding $\cG$ above $(D_{e+1}T_z\cF)(w)$ is $(A,B,1)$-regular and therefore a regular sequence by  corollary \ref{reg-seq}. 
\end{proof}

To finish the proof, we need the following lemma from commutative algebra. This follows by induction on $l$ from the case of two ideals in \cite[Exercise A3.17]{Eis}.

\begin{lemma}
    Let $(I_j)_{j\in [l]}$ be ideals in a commutative ring such that the union of their generators forms a regular sequence in any order. Then $\cap_{j\in [l]} I_j = \prod_{j\in [l]} I_j.$  
\end{lemma}

\subsection{Completing the proof of theorem \ref{rel-sing}} As we mentioned after stating theorem \ref{rel-sing-y}, it implies theorem \ref{rel-sing} for the same degree sequences. The remaining cases to prove for theorem \ref{rel-sing} are when $\ch(\k) = p$ and $e=p+1$ or $e=p+2.$ To prove them we will need a suitable version of lemma \ref{Euler-rel}.

\begin{lemma}
    Suppose theorem \ref{rel-sing} holds for smaller degree sequences. Then there exist constants $A,B$ depending on the degrees such that if $f,\cF$ are homogeneous and $\deg f \in\{p+1,p+2\}$, $\cF$ is $(A,B,t)$-regular and  $\rk_{I(T\cF)} (\nabla_z f) \le t,$ then we have $\rk_{I(\cF)} (f) \le At^B.$
\end{lemma}

\begin{proof}
    If $p$ doesn't divide $\deg f$ we can plug in $z = x$ and get $\rk_{I(\cF)} (f) \le t. $ In the remaining case ($p=2,\deg f =4$) we follow the proof of lemma \ref{Euler-rel} to obtain $\cG = (g_1(x),\ldots,g_r(x))$ such that $T\cF,\cG$ is an $R_3$-sequence and $r\le At^B.$ By the same argument given there, $f$ is reducible in $\k[\bar x]/I(\cF,\cG).$
\end{proof}

\begin{proof}[Proof of theorem \ref{rel-sing} when $\deg(f)\in\{p+1,p+2\}$]
    Inductively, we prove this first for degree $p+1$ and then for degree $p+2.$ The proof we gave for theorem \ref{rel-sing-y} (in the case where there is only one variable appearing) yields
    $\rk_{I(T\cF)} (\nabla_z f) \le Ft^G.$ Then we finish by the above lemma. 
\end{proof}

\section{Theorem \ref{main} for finite fields}

Throughout this section $\k$ will denote a finite field of characteristic $p.$ Our goal is to prove the following result, from which theorem \ref{main} follows.

\begin{theorem}[Regularity over the algebraic closure]\label{rk-closure}
    Let $A,B$ be constants and $\textbf{d}=(d_1,\ldots,d_h)$ a degree sequence with $d_i\le p$ for all $i.$ There exist constants $C(A,B,\textbf{d}),D(A,B,\textbf{d})$ such that if $\cF$ is a $(C,D,t)$-regular tower of degree $\textbf{d}$ in $\poly$ then $\cF$ is $(A,B,t)$-regular as a tower in $\bar\k[x_1,\ldots,x_n].$
\end{theorem}

\begin{proof}[Proof of theorem \ref{main} for finite fields, assuming theorem \ref{rk-closure}]
    Let $A,B$ be the constants of theorem \ref{rel-sing} for degree sequence $\textbf{d} = (1,2,\ldots,d),$ and let $C(A,B,\textbf{d}),D(A,B,\textbf{d})$ be the constants of theorem \ref{rk-closure}. By lemma \ref{regularize}, there exists a $(C,D,t)$-regular tower $\cG$ in $\poly$ with degrees $(1,\ldots,d)$ such that $(g_1,\ldots,g_s)\in I(\cG)$ and $|\cG| \le F(s+t)^G,$ where $F(C,D),G(C,D)$ are constants. By theorem \ref{rk-closure} the tower $\cG$ is $(A,B,t)$-regular as a tower in $\bar\k[x_1,\ldots,x_n],$ which implies it is an $R_t$-sequence by corollary \ref{reg-seq}.  
\end{proof}

The proof of theorem \ref{rk-closure} will be, of course, by induction on the degree sequence. We now assume it holds for lower degree sequences than $\textbf{d} = (d_1,\ldots,d_h)$ and obtain some useful consequences. First, we have a version of lemma \ref{Euler-rel} for finite fields.

\begin{lemma}\label{Euler-rel-fin}
    Assume theorem \ref{rk-closure} holds for lower degree sequences than $(d_1,\ldots,d_h).$ There exist constants $A,B$ depending on the degrees such that if $f\in\poly$ is homogeneous of degree $d_h,$ $\cF$ is an $(A,B,t)$-regular tower in $\poly$ of degrees $(d_1,\ldots,d_{h-1})$ and  $\rk_{I(T\cF)} (\nabla_z f) \le t,$ then we have $\rk_{I(\cF)} (f) \le At^B.$
\end{lemma}

\begin{proof}
     If $d_h < p$ then we can just plug in $z = x$ and get $\rk_{I(\cF)} (f) \le t.$ Otherwise, as in the proof of lemma \ref{Euler-rel} we get a tower $\cG$ in $\poly$ of size $|\cG|\le Ft^G$ and degree $(1,\ldots,p-1)$ such that $(\cF,\cG)$ is a $(C,D,3)$-regular tower in $\poly$ and $\nabla_z f\in I(T\cF,\cG).$ By theorem \ref{rk-closure}, $(\cF,\cG)$ an $(A,B,3)$-regular tower in $\bar\k[x_1,\ldots,x_n]$ and so by corollary \ref{reg-seq} $(\cF,\cG)$ is an $R_3$-sequence. We claim that the image of $f$ in $R := \poly/I(\cF,\cG)$ is reducible, which implies $\rk_{I(\cF)} (f) \le Ft^G+1$ as desired. This will be proved the same way as corollary \ref{Euler-p}. Suppose, to get a contradiction, that the image of $f$ in $R$ is irreducible. Then in particular $R/(f)$ is reduced and so by \cite[\href{https://stacks.math.columbia.edu/tag/030U}{Tag 030U}]{stacks} $\bar\k[x_1,\ldots,x_n]/I(\cF,\cG,f)$ is also reduced. Then the algebraic set $X:= V(\cF,\cG,f)$ has $\codim_{\A^n} X = |\cF|+|\cG|+1$ and ideal \linebreak
     $I(X) = I(\cF,\cG,f).$ $\nabla_z f(x)$ vanishes for any $x\in X$ and $z\in T_x X,$ so the codimension of the tangent space is always $\le |\cF|+|\cG|$ which contradicts the fact that $X$ must contain smooth points. 
\end{proof}

\subsection{Reduction to towers of multi-linear forms}
Recall the tower $\tilde \cF$ of multi-linear forms obtained from $\cF$ which we defined in section 6. The next lemma says that if $\cF$ is sufficiently regular then so is $\tilde\cF.$

\begin{lemma}\label{ml-reg-fin}
    Assume theorem \ref{rk-closure} holds for lower degree sequences than $(d_1,\ldots,d_h)$ and let $A,B$ be constants. There exist constants $C(\textbf{d},A,B),D(\textbf{d},A,B)$ such that if $\cF$ is a $(C,D,t)$-regular tower in $\poly$, then $\tilde\cF$ is an $(A,B,t)$-regular tower over $\k.$
\end{lemma}

\begin{proof}
    This follows from lemma \ref{Euler-rel-fin} the same way that lemma \ref{ml-reg} follows from lemma \ref{Euler-rel}. 
\end{proof}

The converse over $\bar\k$ (and over any field) is also true and is proved easily.

\begin{lemma}\label{polarization}
    Suppose $\cF$ is a tower in $\bar\k[x_1,\ldots,x_n]$ with degrees $(d_1,\ldots,d_h)$ and $d=\max(d_1,\ldots,d_h).$ If $\tilde\cF$ is $(2^dA,B,t)$-regular then $\cF$ is $(A,B,t)$-regular.
\end{lemma}

\begin{proof}
    Suppose $\rk_{I(\cF_{<i})} (\cF_i) \le r$ for some $i,r.$ Applying $\nabla_{v_1}\ldots\nabla_{v_{d_i}},$ we get \linebreak
    $\rk_{I(\tilde\cF_{<i})} (\tilde\cF_i) \le 2^{d_i}r.$
\end{proof}

Given these two lemmas, theorem \ref{rk-closure} will follow from a corresponding theorem for towers of multi-linear forms. To define these multi-linear towers we set up some notation. Let $ V_1,\ldots,V_d$ be finite dimensional vector spaces over $\k.$ For $I\subset [d]$ set $V^I:= \prod_{i\in I} V_i.$

\begin{definition}
     A tower $\cG$ in $\k[V^{[d]}]$ is a multi-linear tower if for all $g_{i,j}\in\cG_i$ there exists some $I(g_{i,j})\in \binom{[d]}{d_i}$ such that $g_{i,j}$ only depends on the $V^I$ coordinates and $g_{i,j}:V^I\to\k$ is multi-linear. For $I\in \binom{[d]}{d_i}$ we write $\cG_i^I$ for the collection of forms $g_{i,j}$ with $I(g_{i,j}) = I.$
\end{definition}

For example $\tilde\cF$ is a multi-linear tower. We can now state the version of theorem \ref{rk-closure} for multi-linear towers. Note that now there is no restriction on the characteristic.

\begin{theorem}\label{rk-closure-ml}
     Let $A,B$ be constants and $\textbf{d}=(d_1,\ldots,d_h)$ a degree sequence. There exist constants $C(A,B,\textbf{d}),D(A,B,\textbf{d})$ such that if $\cG$ is a $(C,D,t)$-regular multi-linear tower of degree $\textbf{d}$ over $\k$ then $\cG$ is $(A,B,t)$-regular as a tower over $\bar\k.$
\end{theorem}

\begin{proof}[Proof of theorem \ref{rk-closure} assuming theorem \ref{rk-closure-ml}]
By lemma \ref{ml-reg-fin}, $\tilde \cF$ is regular over $\k.$ By theorem \ref{rk-closure-ml} it's regular over $\bar\k.$ By lemma \ref{polarization} we get that $\cF$ is regular over $\bar\k.$   
\end{proof}

\subsection{Universality and proof of theorem \ref{rk-closure-ml}}

The proof of theorem \ref{rk-closure-ml} will be based on a universality result for sufficiently regular towers of multi-linear forms.

\begin{theorem}[Universality of regular towers]\label{universal}
    Let $\textbf{d} =(d_1,\ldots,d_h)$ be a degree sequence with $d = \max_{i\in[h]}(d_i)$ and let $N=N_1>\ldots>N_h$ be positive integers. There exist constants $A(\textbf{d}),B(\textbf{d})$ such that the following holds. Suppose $\cG,\cH$ are multi-linear towers in $\k[V^{[d]}]$ of degree $\textbf{d}$ and that for all $i\in[h]$ and $I\in\binom{[d]}{d_i}$ we have $|\cG_i^I| = |\cH_i^I|.$ Suppose also that for all $i\in[h]$ we have $\rk_{I(\cG_{<i})}(\cG_i) > A(N_i(m_i+\ldots+m_h))^B.$ Then for all $j\in[d]$ there exist linear maps $T_j:\k^N\to V_j$  such that for all $i\in[h]$ and $I\in\binom{[d]}{d_i}$ we have
    \[
    \cG^I_i\circ (T_1\times\ldots\times T_d) = \cH_i^I \mod(x_k(l))_{k\in[d],l>N_i}.
    \]
\end{theorem}

To prove this, we need to look more closely at the collection of equations that we'd like the linear maps $T_i$ to satisfy. Assume without loss of generality that $ V_i = \k^{n_i}. $ Then linear maps $ T_i : \k^N\to V_i $ are of the form $ T_i(x) = A_i\cdot x $ where $ A_i\in \mathcal{M}_{n_i\times N}(\k). $ For $g_{i,j}\in \cG_i^I $ and a choice of indices $l\in (N_i)^I,$  the coefficient of $ \prod_{k\in I} x_k(l_k) $ in $ g_{i,j}\circ (T_1\times\ldots\times T_d) $ is given by a multi-linear form
\[ 
C_{i,j}^l: \prod_{k\in I}\mathcal{M}_{n_k\times N}(\k) \to \k, 
\]
namely $ C_{i,j}^l((A_k)_{k\in I}) = \left( g_{i,j}\circ ((A_k)_{k\in I}) \right) ((e_{l_k})_{k\in I}), $ where $e_j$ are the standard basis vectors. This gives us a multi-linear tower $\cC$ in $\k[\prod_{i\in[d]} \mathcal{M}_{n_i\times N}(\k)]$ with the layer $\cC_i$ composed of $C_{i,j}^l$ for all $j\in[m_i]$ and relevant choices of $l.$

\begin{claim}\label{rank-universal}
	For all $i\in[h]$ we have $|\cC_i| = m_i(N_i)^{d_i}$ and $\rk_{I(\cC_{<i})} (\cC_i) \ge \rk_{I(\cG_{<i})} (\cG_i).$ 
\end{claim} 

\begin{proof}
    The statement about $|\cC_i|$ follows easily from our definitions. Now, suppose to get a contradiction, that the rank statement is false. Let $r_i = \rk_{I(\cG_{<i})} (\cG_i).$ Then we have a non-trivial linear combination
    \[
    C = \sum_{j,l} \lambda_j^l C_{i,j}^l
    \]
    with $\rk_{I(\cC_{<i})} (\cC_i) < r_i.$ Then there is some $I_0,l^0$ such that the coefficients $\lambda_j^{l^0}$ are not all zero. Now let $x_k$ be the coordinates on $V_k$ and make the linear substitution 
    \[
    A_k(s,t) = \mathbbm{1}_{t=(l^0)_k} x_k(s).
    \]
    Under this linear substitution, $\rk_{I(\cC_{<i})} (\cC_i) < r_i$ implies
    \[
    \rk_{I(\cG_{<i})} (\sum_{j,g_{i,j}\in\cG_i^{I_0}} \lambda_j^{l^0} g_{i,j}) <r_i
    \]
    which is a contradiction.
    \end{proof}

\begin{corollary}
    Given constants $C,D$ there exist constants $A(C,D,\textbf{d}),B(C,D,\textbf{d})$ such that if $\cG$ satisfies the hpothesis on rank with constants $A,B$ then $\cC$ is $(C,D,t)$-regular. 
\end{corollary}

The desired equalities
\[
    \cG^I_i\circ (T_1\times\ldots\times T_d) = \cH_i^I \mod(x_k(l))_{k\in[d],l>N_i}.
\]
are in fact a system of equations of the form
\begin{equation}\label{universal-eq}
	  C_{i,j}^l(A_1,\ldots, A_d) = \alpha_{i,j}^l,
\end{equation}
where $\alpha_{i,j}^l$ are the coefficients of $h_{i,j}.$ The proof of theorem \ref{universal} is then completed by the following result \cite[lemma 3.2]{rel}.

\begin{theorem}
    There exist constants $C(\textbf{d}),D(\textbf{d})$ such that if $\cG$ is a $(C,D,1)$-regular multi-linear tower in $\k[V]$ then for any choice of scalars $\alpha\in \prod_{i,j}\k,$ there exists a choice of $x\in V(\k)$ with $g_{i,j}(x) = \alpha_{i,j}.$ In other words, the map \linebreak $\cG:V(\k)\to \prod_{i,j} \k$ is surjective. 
\end{theorem}

Now we are ready to prove theorem \ref{rk-closure-ml}.
\begin{proof}
    Let $r_j = A(|\cG_{\ge j}|+t)^B+1,$ and apply theorem \ref{universal} with $N_i = \sum_{j\ge i} 2 m_j r_j.$ Note that if $\cG$ is $(C,D,t)$-regular for sufficiently large constants $C(A,B,\textbf{d}),D(A,B,\textbf{d})$ then it satisfies the hypotheses of theorem \ref{universal} for these values of $N_i.$ Now we choose the tower $\cH$ that we obtain after linear substitution. First, identify \linebreak
    $\k^N\cong \prod_{i\in [h]} \mathcal{M}_{m_i\times 2r_i}(\k).$ Thus for $k\in[d]$ and $i\in[h],$ $x_k(i)\in \mathcal{M}_{m_i\times 2r_i}(\k).$ For each $g_{i,j}\in \cG_i^I,$ we choose $h_{i,j}\in \cH_i^I$ as follows
    \[
    h_{i,j} ((x_k)_{k\in I}) = \sum_{s\in[2r_j]} \prod_{k\in I} (x_k(i))_{j,s}.
    \]
    By theorem \ref{universal}, for all $j\in[d]$ there exist linear maps $T_j:\k^N\to V_j$  such that for all $i\in[h]$ and $I\in\binom{[d]}{d_i}$ we have
    \[
    g_{i,j}\circ (T_1\times\ldots\times T_d) = h_{i,j} \mod(x_k(l))_{k\in[d],l<i}.
    \]
    Let $T = T_1\times\ldots\times T_d.$ Suppose, to get a contradiction, that over $\bar\k$ we have $\rk_{I(\cG_{<i})} (\cG_i) < r_i$ for some $i\in[h].$ Then, of course, $\rk_{I(\cG_{<i}\circ T)}(\cG_i \circ T) < r_i.$ Restricting to the subspace $\{x_k(l) = 0: k\in[d],l < i\},$ we obtain $\rk(\cH_i) < r_i.$ So there exists a non-trivial linear combination with  $\rk(\sum_j \lambda_j h_{i,j}) <r_i. $ Suppose without loss of generality that $\lambda_1 \neq 0$ and restrict further to the subspace $\{x_k(i)_{j,s} = 0: k\in[d], j > 1\}$ to get $\rk(h_{i,1}) < r_i.$ By lemma \ref{low-rk-sing} this implies that $S(h_{i,1})$ has codimension $<2r_i,$ but this is easily seen to be a union of subspaces of codimension $2r_i,$ contradiction.
\end{proof}

\bibliography{ref}
\bibliographystyle{plain}

\end{document}